\documentclass{amsart}
\usepackage{mathrsfs} \usepackage{mathabx}
\usepackage[varg]{txfonts} 
\include{amsmath} \include{amssymb}

\newcommand{\equals}{\approx}

\newcommand{\Mod}{\mathsf{Mod}}
\newcommand{\Diag}{\mathsf{Diag}}
\newcommand{\Dom}{\mathsf{Dom}}
\newcommand{\Horn}{\forall\mathsf{Horn}}
\newcommand{\Identities}{\mathsf{Identities}}
\newcommand{\Distinct}{\mathsf{Distinct}}
\newcommand{\bbZ}{\mathbb{Z}}
\newcommand{\NCR}{\mathsf {NCR}_1}
\newsymbol \tfore 1329

\newcommand{\scrF}{{\mathscr F}}

\newcommand{\scrP}{{\mathscr P}}
\newcommand{\scrS}{{\mathscr S}}

\newcommand{\bP}{\mathbf{P}}
\newcommand{\bQ}{\mathbf{Q}}
\newcommand{\bfx}{\mathbf{x}}
\newcommand{\bfp}{\mathbf{p}}

\newcommand{\LT}{{\sf LT}}
\newcommand{\tuple}[1]{{ \bf{#1}}}

\newtheorem{daf}{Definition}

\newtheorem{thr}[daf]{THEOREM}

\newtheorem{lam}[daf]{LEMMA}

\newtheorem{cllr}[daf]{COROLLARY}

\newtheorem{rmk}[daf]{Remark}

\newtheorem{exmp}[daf]{Example}

\newtheorem{clm}{Claim}

\newtheorem{fct}[daf]{Fact}

\newtheorem{prb}{Problem}

 \newtheorem{theorem}{Theorem}[section]

\newtheorem{lemma}[theorem]{Lemma}

\newtheorem{definition}[theorem]{Definition}
\newtheorem{example}[theorem]{Example}

\newtheorem{question}{Question}

\newcounter{dflistctr}
{\end{list}}
\newcounter{txlistctr}
{\end{list}}
\newcounter{thlistctr}
\newenvironment{thlist}{\ 
\begin{list}%
{\alph{thlistctr}}%
{\setlength{\labelwidth}{2ex}%
\setlength{\labelsep}{1ex}%
\setlength{\leftmargin}{6ex}%
\usecounter{thlistctr}}}%
{\end{list}}
\newcounter{thlistctr:p}
\newenvironment{thlist:p}{\ 
\begin{list}%
{\alph{thlistctr:p}}%
{\setlength{\labelwidth}{2ex}%
\setlength{\labelsep}{1ex}%
\setlength{\leftmargin}{6ex}%
\usecounter{thlistctr:p}}}%
{\end{list}}
\newcounter{remlistctr}
{\end{list}}

\newtheorem{prp}[daf]{PROPOSITION}

\title[Boole's Principles]{Boole's Principles of Symbolical Reasoning }

\author{Stanley Burris}
\address{Department of Pure Mathematics\\
University of Waterloo\\
Waterloo, Ontario, N2L 3G1, Canada}
\email{snburris@math.uwaterloo.ca}

\author{H.P. Sankappanavar}
\address{Department of Mathematics\\
SUNY at New Paltz\\
New Paltz, New York 12561, USA}
\email{sankapph@newpaltz.edu}
\date{\today}  

\begin{document}

\begin{abstract}
In modern algebra it is well-known that one cannot, in general, apply ordinary equational reasoning when dealing with partial algebras. However Boole did not know this, and he took the opposite to be a fundamental truth, which he called the Principles of Symbolical Reasoning in his 1854 book {\em Laws of Thought}. Although Boole made no mention of it, his Principles were clearly a generalization of the earlier publications on algebra by the Cambridge mathematician Peacock. After a detailed examination of Boole's presentation of his Principles, we give a correct version that is applicable to his algebra of logic for classes.
\end{abstract}

\maketitle


When Boole started his mathematical research in the late 1830s, he was particularly fascinated by two tools recently introduced into mathematics, the use of {\em operators} (to solve differential equations), and the {\em symbolical method} (which justified algebraic derivations even though the intermediate steps were not interpretable).  Differential operators had been introduced by the French and used by prominent French mathematicians such as Cauchy; but by 1830 the French had pretty much abandoned them---Cauchy felt that not enough was known about them to ensure that their use led to correct results. 

English mathematicians, especially Boole, took up differential operators with great enthusiasm, and in 1844 he won a gold medal from the Royal Society for a paper based on them. When he turned to create his algebra of logic for classes in 1847, it was natural for Boole to build it on a foundation of operators, namely selection operators. For example, the expression ``big shiny '' was viewed as the operator $y$ that selected shiny objects from a class,  followed by the operator $x$ that selected ``big'' objects. The composition of the two operators was written simply as $xy$, and, as with differential operators, $xy$ was viewed as the product of $x$ and $y$. It only remained for Boole to find appropriate definitions for addition and subtraction.  In the 1854 version of his algebra of logic,  Boole replaced the symbols for selection operators with the symbols for classes, relegating the selection operators to a footnote on operations of the mind.

The symbolical method was introduced by Peacock \cite{Peacock-1830} in his 1830 {\em Treatise on Algebra}---this book gave his resolution of concerns about negative and complex numbers. (The {\em Treatise} was expanded into two volumes in 1842/1845.) He split algebra into two parts: (1) {\em Arithmetical Algebra}, which was the algebra of positive numbers (which meant the operation of subtraction was a partial operation), and (2) {\em Symbolical Algebra}, which dispensed with interpretations and dealt solely with equations that could be derived from basic laws.  

The two kinds of algebra were tied together by Peacock's fundamental principle of {\em The Permanence of Equivalent Forms}, which basically said that general truths of the Arithmetical Algebra provided the laws of Symbolical Algebra, and equations derived from the laws in the latter gave true facts about positive numbers whenever they applied to them. 
Arithmetical Algebra had a partial algebra as its interpretation, whereas Symbolical Algebra simply carried out derivations as though the operations were total, without the benefit of an interpretation. Thus $\sqrt{-1}$ was no longer a mysterious number whose nature was to be debated---it was just a symbolical expression devoid of meaning.

Boole was able to create an algebra of logic for classes that had laws remarkably close to those of ordinary algebra;\footnote{
As strange as Boole's use of partial algebras might look to the uninitiated, it seems quite possible that he may have found the simplest way to incorporate the laws he used (stated in $\S$\ref{B text} below) into an algebra of logic for classes.
} 
  but this came at a cost, namely his operations of $+$ and $-$ were partial operations. Since he wanted to freely use algebraic reasoning that was not impeded by the possibility that terms in his derivations might not be interpretable (as classes), he formulated a general version of Peacock's approach to algebra. Boole's {\em Principles of Symbolical Reasoning} lifted Peacock's work out of the confines of numbers to the general setting of partial algebras. Briefly stated, Boole's Principles said that given a collection $\scrP$ of partial algebras and a collection $\Sigma$ of laws that $\scrP$ satisfied (that is, satisfied in all instances for which the terms in the laws were defined),  one could carry out equational reasoning as though one were working with total algebras to obtain valid results for the class $\scrP$. 

Boole's Principles are clearly far too general. Here is  a simple example to show that the equational logic for total algebras (see \cite{BuSa-1981}) need not be valid for partial algebras.
Let $\mathbf{P} =\langle \{0, 1 \},+\rangle$ be the partial algebra given by $0+0=0$ and $1+1=1$; and otherwise $+$ is undefined.   Then the equations $x+y \approx x$ and $x+y \approx y$ hold in $\bP$  whenever $x+y$ is defined. But  $x \approx y$, an equational logic consequence of these two equations, does not hold in $\bP$.

Boole's Principles were presented in Chapter V of his 1854 book {\em Laws of Thought} (this book will be referred to as \LT\  in this paper).
His somewhat rambling introduction to these Principles first started to come into focus with the statement of them on p.~67; his final analysis came on p.~69, where he  said that to justify his Principles one only needed a single example.\footnote
{This was based on his claim that in logic, unlike in the natural sciences, one only needed a single example to justify a general result.}
The only example (besides his algebra of logic for classes) that he mentioned with regard to his Principles was that of the trigonometric identities derived using the imaginary number $\sqrt{-1}$. Evidently Boole was not convinced that this was the desired single example since he also said that such derivations of trigonometric identities could evidently only be justified by appealing to his Principles. He rounded out the defence of his Principles by saying that they should simply be accepted as {\it fundamental facts} about knowledge. 

In $\S$\ref{B text} below, the long-neglected 1854 discussion by Boole of his Principles is carefully examined. $\S$\ref{New Thm} gives a corrected version (Theorem \ref{principles}) of these Principles, based on universal Horn sentences with relativized quantifiers. $\S$\ref{applic} shows that Theorem \ref{principles} does indeed apply to Boole's algebra of logic for classes, thanks to the ground-breaking book of Hailperin \cite{Hailperin-1976} on Boole's work, and to Gorbunov's analysis of quasi-identity logic \cite{Gorbunov-1998}.\footnote
{We are indebted to our colleagues, Professors Kira Adaricheva and Anvar Nurakunov,
for this reference, and to Professor George McNulty for discussions on this topic.
}  

\section{An analysis of Boole's text on his Principles} \label{B text}
Boole claimed, in \LT, that the symbols of common algebra were the natural ones to use for an algebra of logic for classes, and it was a happy coincidence that the laws of common algebra agreed with the laws of his algebra of logic.\footnote{
 In 1999 Priest \cite{Priest-1999} noted that the publication of a significant portion of Boole's Nachlass in \cite{G-G-1997} clarified the fact that ``Boole's driving inspiration was the analogy between arithmetic and logic''. It is easier to sort out the text of \LT\ if one assumes that Boole started with the goal of using common algebra, with its laws and rules of inference, and shoe-horned the logic of classes into this framework (with the help of a restricted use of the idempotent law). 
} 
Boole's algebras $\bP(U) = \langle P(U),+,\cdot,-,0,1 \rangle$, where $P(U)$ is the collection of subsets of the universe $U$, had $0 := \O$, $1:= U$, and multiplication defined as intersection; both addition and subtraction were only partially defined. Indeed, $A+B := A\cup B$ and $A-B := A \setminus B$, when defined, where the domains of addition and subtraction were given by:
\begin{eqnarray*}
\Dom(+)&=& \{(A,B) \in P(U)^2 : A\cap B = 0 \}\\
\Dom(-)&=& \{(A,B) \in P(U)^2 : B\subseteq A \}.
\end{eqnarray*}
Since some of the operations of Boole's algebras were only partially defined, Boole was clearly working with partial algebras.

Boole found laws of the $\bP(U)$,  such as $x+y \approx y+x$, and $x^2 \approx x$,  by looking at the instances where the terms of the laws were defined. After a careful study of Boole's \LT, Hailperin \cite{Hailperin-1976} determined that the collection of laws actually used by Boole was
\begin{itemize}
\item
 $\NCR$, the usual laws for commutative rings with unity, and laws excluding additively nilpotent elements ($nx \approx 0 \Rightarrow x \approx 0$, for $n$ a non-zero integer),
 \item
 $0\napprox 1$, and
 \item
 the special law $x^2\equals x$, which was only applied to class-symbols $A,B,\ldots$.
\end{itemize}

In developing his algebra of logic for classes, Boole reasoned with these laws as if the operations were total, thus using reasoning which, in general, is not valid for partial algebras.  He would start with ground premiss equations whose terms were defined for all values of the class-symbols, and end up with conclusion equation(s) that likewise had terms that were totally defined. But in the intermediate steps in a derivation of the conclusion(s) he could use equations that had terms that were uninterpretable, that is, only partially defined. This was likely the major stumbling block to the understanding of his algebra of logic by mathematicians interested in working with a symbolic logic.

Boole presented his defence of the use of uninterpretables in Chapter V of \LT, which is titled 
\begin{quote} \sf
Of the fundamental principles of symbolical reasoning, and 
of the expansion or development of expressions involving 
logical symbols.
\end{quote}
The first half of the title is the subject of the first six items of Chapter V, where Boole attempted to allay the reader's concerns about the legitimacy of using uninterpretables  in the intermediate steps of a derivation.  These six items are carefully examined, one in each of the the following six subsections,  to see to what extent Boole succeeded---the quotes from Chapter V of \LT\ given below are presented as indented paragraphs in sans serif font.

\subsection{On item 1 of Chap.~V of \LT}
In the first item of Chapter V Boole said that so far he had set up the notation for an algebra of logic, described the fundamental operations that he would use, determined the laws, and had shown how to render primary propositions (his version of categorical propositions) about classes as equations. Next he wanted to develop his algebra of logic by proving theorems and finding algorithms; but first he needed to say something about how he was going to do this. Actually the reader would have to wait till item 4 of Chapter V before learning that Boole would simply be using ordinary equational reasoning on  his partial algebras.

\begin{quote} \sf
1. The previous chapters of this work have been devoted to 
the investigation of the fundamental laws of the operations 
of the mind in reasoning; of their development in the 
laws of the symbols of Logic; and of the principles of expression, 
by which that species of propositions called primary may be represented 
in the language of symbols. These inquiries have been 
in the strictest sense preliminary. They form an indispensable 
introduction to one of the chief objects of this treatise---the construction 
of a system or method of Logic upon the basis of an 
exact summary of the fundamental laws of thought. There are 
certain considerations touching the nature of this end, and the 
means of its attainment, to which I deem it necessary here to 
direct attention. 
\end{quote}

\subsection{On item 2 of Chap.~V of \LT}
In the second item Boole started by saying that in order to have a general method for dealing with the logic of classes he needed general laws and rules of inference. He noted that addition was a partial operation (which presumably could be an obstacle to creating a general method).
\begin{quote}\sf
2. I would remark in the first place that the generality of a 
method in Logic must very much depend upon the generality of 
its elementary processes and laws. one has, for instance, in the 
previous sections of this work investigated, among other things, 
the laws of that logical process of \textit{addition} which is symbolized 
by the sign $+$. Now those laws have been determined from the 
study of instances, in all of which it has been a necessary condition, 
that the classes or things added together in thought should 
be mutually exclusive.  The expression $x + y$ seems indeed uninterpretable, 
unless it be assumed that the things represented by $x$ and the things represented by $y$ are entirely separate; that they embrace no individuals in common.
\end{quote}
Boole never satisfactorily explained in his publications why it was necessary to restrict addition $A+B$ to disjoint classes $A,B$.  The real reason for this restriction on the definition of addition was surely Boole's desire to use ordinary algebra for an algebra of logic (which he was able to do by adding restricted use of the idempotent law). Then, if $A+B$ were defined it would be idempotent, that is, $(A+B)^2 = A+B$. Ordinary algebra, along with the idempotence of $A$ and $B$, would lead to $2AB = 0$, and thus to $AB=0$; this meant $A$ and $B$ were disjoint. Boole actually gave the details of deriving $AB=0$ from $(A+B)^2 = A+B$ in his unpublished manuscripts---see p.~21 of \cite{Boole-1952}.

\begin{quote} \sf
 And conditions 
analogous to this have been involved in those acts of conception 
from the study of which the laws of the other symbolical operations 
have been ascertained. 
\end{quote}
The only other operation that was restricted was subtraction.  The (unstated) reason for this restriction is again determined by looking at the consequences of $A-B$ being idempotent, as Boole also noted in the aforementioned unpublished manuscript. Next Boole stated the potential Achilles heel of his presentation.
\begin{quote} \sf
The question then arises, whether
it is necessary to restrict the application of these symbolical laws 
and processes by the same conditions of interpretability under 
which the knowledge of them was obtained. 
\end{quote}
Boole proceeded to try to make the case that the answer was always ``no restrictions needed'', saying that otherwise his program for an algebra of logic must fail.
\begin{quote} \sf
If such restriction 
is necessary, it is manifest that no such thing as a general 
method in Logic is possible. 
\end{quote}
By a `general method' Boole meant algorithms to deduce desired equational consequences from equational premisses. Actually it was {\em not} manifest that a general method would be impossible under the restriction that all steps only use interpretable terms; it was just that a general method would, in some cases, under such restrictions, be rather unwieldy compared to the method Boole presented. His answer as to whether or not such a restriction was necessary would wait till the end of item 3 below.
\begin{quote} \sf
On the other hand, if such restriction is unnecessary, in what light are we to contemplate processes
which appear to be uninterpretable in that sphere of thought 
which they are designed to aid? These questions do not belong 
to the science of Logic alone. They are equally pertinent to every 
developed form of human reasoning which is based upon the 
employment of a symbolical language. 
\end{quote}

\subsection{On item 3 of Chap.~V of \LT}
 Boole started this item by noting that in everyday reasoning one did not employ uninterpretable steps.
  \begin{quote} \sf
3. I would observe in the second place, that this apparent 
failure of correspondency between process and interpretation does 
not manifest itself in the \emph{ordinary} applications of human reason.
For no operations are there performed of which the meaning
and the application are not seen; and to most minds it does 
not suffice that merely formal reasoning should connect their 
premises and their conclusions; but every step of the connecting 
train, every mediate result which is established in the course of 
demonstration, must be intelligible also. And without doubt, 
this is both an actual condition and an important safeguard, in 
the reasonings and discourses of common life. 
\end{quote}
Next he said that there are perhaps those who would like to apply the same requirement,
of every step being meaningful, to symbolical arguments. 
Eventually he would claim, in item 4, that it was enough that the premisses and conclusion were meaningful in order to obtain a valid, meaningful argument.
\begin{quote} \sf
There are perhaps many who would be disposed to extend 
the same principle to the general use of symbolical language as 
an instrument of reasoning. It might be argued, that as the 
laws or axioms which govern the use of symbols are established 
upon an investigation of those cases only in which interpretation 
is possible, one has no right to extend their application to other 
cases in which interpretation is impossible or doubtful, even 
though (as should be admitted) such application is employed in 
the intermediate steps of demonstration only. 
\end{quote} 
Next he repeated his belief that the symbolical method
offered little to the study of logic unless one could use the laws of the partial algebras freely.
\begin{quote} \sf
Were this objection conclusive, it must be acknowledged that slight advantage
would accrue from the use of a symbolical method in
Logic. Perhaps that advantage would be confined to the mechanical gain of employing short and convenient symbols in the place of more cumbrous ones. 
\end{quote}
The phrase `it must be acknowledged' really meant that Boole believed 
there would be little gain in developing an algebra of logic if one were required to 
make terms interpretable in every step.
At this point the reader would be justified in expecting Boole to offer a profound insight
concerning uninterpretable steps in symbolical reasoning. 
Instead one simply hears the voice of authority.
\begin{quote} \sf
But the objection itself is fallacious. 
Whatever our \textit{\`{a} priori} anticipations might be, it is an 
unquestionable fact that the validity of a conclusion arrived at
by any symbolical process of reasoning, does not depend upon 
our ability to interpret the formal results which have presented 
themselves in the different stages of the investigation. 
\end{quote}
The assertion Boole makes about `an unquestionable fact' is wrong. It is 
simply an erroneous belief that Boole firmly held; his 
 algebra of logic would need to be completely reworked without this `fact'. 
\begin{quote} \sf
There 
exist, in fact, certain general principles relating to the use of 
symbolical methods, which, as pertaining to the particular subject 
of Logic, I shall first state, and I shall then offer some remarks 
upon the nature and upon the grounds of their claim to 
acceptance. 
\end{quote}
The reader will search in vain for a clear statement of any actual grounds for acceptance of the Principles in Boole's text. It seems that the use of the Principles to prove general results about complex numbers and about logic, results which were sound in all the examples Boole had checked, were his only grounds.

\subsection{On item 4 of Chap.~V of \LT}
In this item Boole laid out the requirements of symbolical reasoning that he believed were 
sufficient to guarantee the validity of the results whenever they were interpretable.

\begin{quote} \sf
4. The conditions of valid reasoning, by the aid of symbols, 
are---

1st, That a fixed interpretation be assigned to the symbols 
employed in the expression of the data; and that the laws of the 
combination of those symbols be correctly determined from that 
interpretation. 
\end{quote}
The first condition said that one was to work with a fixed collection of partial algebras, 
and that the laws one worked with 
actually held whenever the terms of the laws were defined in the partial algebras.

\begin{quote} \sf
2nd, That the formal processes of solution or demonstration 
be conducted throughout in obedience to all the laws determined 
as above, without regard to the question of the interpretability 
of the particular results obtained. 
\end{quote}
Clearly the part of the 2nd requirement that said `don't worry about interpretability' is the contentious point in Boole's conditions;  he would have more to say about this in item 5 of Chap.~V.
\begin{quote} \sf
3rd, That the final result be interpretable in form, and that 
it be actually interpreted in accordance with that system of interpretation 
which has been employed in the expression of the 
data. Concerning these principles, the following observations 
may be made. 
\end{quote}
The 3rd condition said that the conclusion needed to be interpretable (so that one
had a meaningful/useful result).

\subsection{On item 5 of Chap.~V of \LT}

Here Boole reflected on the naturalness of the first and third condition, but noted 
that the second condition likely needed ``a few additional words''.
\begin{quote} \sf
5. The necessity of a fixed interpretation of the symbols has 
already been sufficiently dwelt upon (II. 3). The necessity that 
the fixed result should be in such a form as to admit of that interpretation 
being applied, is founded on the obvious principle, 
that the use of symbols is a means towards an end, that end 
being the knowledge of some intelligible fact or truth. And 
that this end may be attained, the final result which expresses 
the symbolical conclusion must be in an interpretable form. It 
is, however, in connexion with the second of the above general 
principles or conditions (V. 4), that the greatest difficulty is 
likely to be felt, and upon this point a few additional words are 
necessary. 
\end{quote}
What followed was a somewhat confused attempt by Boole to justify 
his Principles. He said they rested on another fact that he had
become aware of---that whereas the natural sciences required many
observations to deduce a law of nature, in logic it was different.
He said that a single clear example made the general principle
known.  
\begin{quote} \sf
I would then remark, that the principle in question may be 
considered as resting upon a general law of the mind, the knowledge 
of which is not given to us \textit{\`{a} priori}, i.e. antecedently to 
experience, but is derived, like the knowledge of the other laws 
of the mind, from the clear manifestation of the general principle 
in the particular instance. 
\end{quote}
No evidence for this sweeping claim was provided, just a continuation of 
the claim.  In reality, fundamental principles
of reasoning are based on their acceptance by the community of 
scholars. If Boole had only said that he was proposing that his 
Principles be accepted as fundamental, his writing
 style in these sections would have been more agreeable.
\begin{quote} \sf
A single example of reasoning, in 
which symbols are employed in obedience to laws founded upon 
their interpretation, but without any sustained reference to that 
interpretation, the chain of demonstration conducting us through 
intermediate steps which are not interpretable, to a final result 
which is interpretable, seems not only to establish the validity of 
the particular application, but to make known to us the general 
law manifested therein. No accumulation of instances can properly 
add weight to such evidence. It may furnish us with clearer 
conceptions of that common element of truth upon which the application 
of the principle depends, and so prepare the way for its 
reception. It may, where the immediate force of the evidence is 
not felt, serve as a verification, \textit{\`{a} posteriori}, of the practical validity 
of the principle in question. But this does not affect the position 
affirmed, viz., that the general principle must be seen in the 
particular instance,---seen to be general in application as well as 
true in the special example. 
\end{quote}
Now it seemed that Boole was ready for the coup de gr\^{a}ce, to give that single example to show that, 
according to his `one example is enough' thesis, his Principles were sound.
He stated what is likely the only example he knew of where mathematicians
 worked with
uninterpretables, namely the use of
the uninterpretable $\sqrt{-1}$ when working with numbers. As an interesting
example of results obtained by using $\sqrt{-1}$ he mentioned  trigonometric identities.\footnote
{ This likely refered to results such as DeMoivre's Theorem, that
$$
(\cos \theta + i\sin \theta)^n \ =\ \cos(n\theta) + i \sin(n\theta).
$$
This example would have been quite difficult for Boole to set up for his Principles, presumably starting
with an algebra on the reals. What would the fundamental operations be that allowed
one to discuss $\cos$ and $\sin$? and what would the laws be that led to a proof of 
DeMoivre's Theorem?
} 
He started by saying this was an example of what had been said---unfortunately it does not seem to be an example for `one example is enough', but rather just another application of his Principles.
\begin{quote} \sf
The employment of the uninterpretable 
symbol $\sqrt{-1}$, in the intermediate processes of trigonometry, 
furnishes an illustration of what has been said. I apprehend that 
there is no mode of explaining that application which does not 
covertly assume the very principle in question. 
\end{quote}
Thus Boole ended up not offering the single example to justify his Principles; instead he offered two examples, common algebra and his algebra of logic,
where he believed his Principles applied.\footnote{
 One wonders why Boole bothered bringing up his `single example suffices' assertion---was he hoping that someone else would provide the one example needed? or perhaps that the mathematical community would say that the application of complex numbers to trigonometric identities was valid without reference to his Principles, and hence could be used as the one example?}
Next one sees Boole simply claiming that his Principles deserved to be accepted as fundamental facts.
He could have replaced the totality of items 1--5 in Chap.~V with simply stating his Principles and making the next statement, leaving it to the reader to decide whether or not to accept them.
\begin{quote} \sf
But that principle, 
though not, as I conceive, warranted by formal reasoning 
based upon other grounds, seems to deserve a place among those 
axiomatic truths which constitute, in some sense, the foundation 
of the possibility of general knowledge, and which may properly 
be regarded as expressions of the mind's own laws and constitution. 
\end{quote}
Of course Boole's Principles did not take hold; they seem to have quietly vanished. 
Many mathematicians and logicians, starting with Cayley and Jevons, did not like Boole's uninterpretables. However we are not aware of any objections prior to 1999 to Boole's belief that ``one example is enough'' to justify his Principles. In 1999 Priest \cite{Priest-1999} wrote that ``Even in logic, the truth of a general rule cannot be simply read off from a particular case''. We are not aware of any publication that has addressed the issues with Boole's Principles. Mathematicians, following Jevons and Peirce, soon side-stepped the issues by modifying Boole's algebra of logic, replacing Boole's addition by union and his subtraction by complement, so that there were no uninterpretables.

\subsection{On item 6 of Chap.~V of \LT}
This item of \LT\ offers more puzzling comments by Boole. He said that the Principles would
be used in \LT\ in the following manner: when carrying out an argument
in his algebra for the logic of classes, if one encountered a step that had uninterpretable terms
then one was to stop thinking about classes and switch to thinking about the step 
as applying to the algebra of 0 and 1 described in his Rule of 0 and 1.\footnote
{
This rule is described in detail in \cite{Burris-SEP}.
}
 When one eventually returned to steps where the terms were interpretable in the logic of classes, then one switched back to thinking about classes.

It seems the only reason for using the algebra of 0 and 1 for the steps with uninterpretable terms in the logic of classes was to help the user remember the laws and procedures  that one could use. Otherwise it played no role.
\begin{quote}\sf
6. The following is the mode in which the principle above 
stated will be applied in the present work. It has been seen, 
that any system of propositions may be expressed by equations 
involving symbols $x$, $y$, $z$, which, whenever interpretation is possible, 
are subject to laws identical in form with the laws of a system 
of quantitative symbols, susceptible only of the values $0$ and 
$1$ (II. 15). But as the formal processes of reasoning depend only 
upon the laws of the symbols, and not upon the nature of their 
interpretation, we are permitted to treat the above symbols, 
$x$, $y$, $z$, as if they were quantitative symbols of the kind above 
described. \textit{We may in fact lay aside the logical interpretation of 
the symbols in the given equation; convert them into quantitative symbols,
susceptible only of the values $0$ and $1$; perform upon them as such 
all the requisite processes of solution; and finally restore to them their 
logical interpretation.} And this is the mode of procedure which 
will actually be adopted, though it will be deemed unnecessary 
to restate in every instance the nature of the transformation employed.
\end{quote}
Next Boole reminded us, again, of how important he thought it was to be able to
work as though one had a total algebra, for he believed that otherwise
the quest for the desired algorithms (to derive certain kinds of conclusions)
would be hopeless.
\begin{quote} \sf
The processes to which the symbols $x$, $y$, $z$, regarded 
as quantitative and of the species above described, are subject, are 
not limited by those conditions of thought to which they would, 
if performed upon purely logical symbols, be subject, and a freedom of operation is given to us in the use of them, without 
which, the inquiry after a general method in Logic would be a 
hopeless quest. 
\end{quote}
Boole concluded this section by saying he had a general method to convert
any equational conclusion into an equivalent one that was interpretable, and that would be
presented next.
\begin{quote} \sf
Now the above system of processes would conduct us to no 
intelligible result, unless the final equations resulting therefrom 
were in a form which should render their interpretation, after 
restoring to the symbols their logical significance, possible. 
There exists, however, a general method of reducing equations 
to such a form, and the remainder of this chapter will be devoted 
to its consideration. 
\end{quote}
This general method was based on Boole's Expansion Theorem, an analog of the disjunctive normal form used in modern Boolean algebra (see \cite{Burris-SEP} for details).


\section{Correcting Boole's Principles}  \label{New Thm}

Throughout this section it is assumed that: (i) all partial algebras being discussed belong to a fixed language $\scrF$; (ii) partial algebras are denoted by capital bold letters $\bP$, $\bQ$, etc. (iii)   $\tuple{x}$ is the list $x_1,\ldots,x_m$, and $\tuple{A}$ is the list $A_1,\ldots, A_m$; (iv) $(\forall \tuple{x}\,)$ means $(\forall x_1)\cdots (\forall x_m)$; (v)  $t:=t(\tuple{x}\,)$ denotes a term. 

Given a partial algebra $\bP$, a term $t$ defines a partial operation $t^\bP$ on $\bP$ with domain $\Dom(t^\bP)$. The domain $\Dom_\bP(\omega)$ of an open formula $\omega(\bfx)$ is the intersection of the $\Dom(t^\bP)$, for $t$ a term appearing in $\omega$. The relation $\omega^\bP $ defined by $\omega$ is the collection of $\bfp \in \Dom_\bP(\omega)$ such that $\omega(\bfp)$ is true in $\bP$.  The following notion of subalgebra will be used when working with partial algebras.

 \begin{definition}
Define $\bP \sqsubseteq \bQ$ 
if $P\subseteq Q$ and the operations of $\bP$, where defined, agree with those of $\bQ$, that is
\begin{equation}
f^\bP(\bfp) = p\quad \text{implies}\quad f^\bQ(\bfp) = p,
\end{equation}
for $\bfp \in \Dom(f^\bP)$.
\end{definition}

One easily sees that the following hold.

\begin{lemma} \label{Elev} Suppose $\bP \sqsubseteq \bQ$.  
\begin{thlist}
\item
If $\tuple{p} \in \Dom(t^\bP)$ then $\tuple{p} \in \Dom(t^\bQ)$ and
\begin{equation*}
t^\bQ(\tuple{p}) = t^\bP(\tuple{p}).
\end{equation*}
\item
For $\omega(\bfx)$ an open formula with $\tuple{p} \in \Dom_\bP(\omega)$, one has 
$\tuple{p} \in \Dom_\bQ(\omega)$ and
\begin{equation}\label{Elev A}
\bP \models \omega(\tuple{p}) \quad \text{iff}\quad \bQ \models \omega(\tuple{p}).
\end{equation}
\end{thlist}
\end{lemma}

\begin{definition}
Given two  partial algebras $\bP$ and $\bQ$, a mapping $\alpha : P \rightarrow Q$ is an embedding of\, $\bP$ into $\bQ$  iff  $\alpha$ is 1-1 and for each fundamental operation $f$,
if  $\,\tuple{p} \in \Dom(f^\bP)$ then $\alpha(\tuple{p}) \in \Dom(f^\bQ)$ and
$$
f^\bQ(\alpha \tuple{p}) = \alpha f^\bP(\tuple{p}).
$$
\end{definition}

Let $\Horn$ be the set of universal Horn sentences belonging to the given language of  partial algebras, let $\Sigma$ be a subset of $\Horn$, let $\delta(x)$ be a conjunction of  atomic formulas, and
let $\scrP$ be a collection of partial algebras $\bP = \langle P,\scrF \rangle$. 
 The following conventions are adopted, where $\sigma\in\Horn$, say $\sigma := (\forall \tuple{x}\,)\omega(\tuple{x}\,)$
 with $\omega := \omega(\tuple{x}\,)$ an open Horn formula:
\begin{thlist}
\item
Whenever a property $\Pi$ of partial algebras $\bP$ is stated for $\scrP$, this means it applies to all members of $\scrP$.

\item
$\omega$ is $\bP$-total if $\Dom_\bP(\omega) = P^m$, and it is $\scrP$-total if it is total for all $\bP\in \scrP$.  Define $\Dom_\bP(\sigma) := \Dom_\bP(\omega)$, etc.

\item
$\sigma$ holds in a partial algebra $\bP$ if $\omega(\tuple{p}\,)$ is true in $\bP$ for all
$\tuple{p} \in \Dom(\sigma)$. If so, we also say $\bP$ satisfies $\sigma$, abbreviated as $\bP \models \sigma$. $\bP \models \Sigma$ means $\bP \models \sigma$ for $\sigma\in \Sigma$.    $\scrP \models \sigma$ if $\bP \models \Sigma$ for all $\bP \in \scrP$.

\item
A law of $\scrP$ is any $\sigma$ such that $\scrP \models \sigma$.

\item
$\sigma|_\delta$ is the universal Horn sentence obtained by relativizing the quantifiers of $\sigma$ to $\delta(x)$.

\item
$\Mod(\Sigma)$ is the collection of total algebras $\bP$ satisfying $\Sigma$.

\item
$\bP$ embeds in $\Mod(\Sigma)$, written $\bP {\hookrightarrow} \Mod(\Sigma)$, if there is a $\bQ \in \Mod(\Sigma)$ and an embedding $\alpha : \bP \hookrightarrow \bQ$. $\scrP \hookrightarrow \Mod(\Sigma)$ if $\bP \hookrightarrow \Mod(\Sigma)$ for every $\bP\in\scrP$.

\item
$\Sigma \vdash \sigma$ means there is a derivation of $\sigma$ from $\Sigma$ in first-order logic.

\item
$\Diag^+(\bP)$ is the set of formulas $f(\bfp) \equals p$, where $\bfp \in \Dom(f^\bP)$ and $p \in P$ are such that
$f^\bP(\bfp) = p$. (This describes the tables for the fundamental operations of $\bP$.)
\end{thlist}
With this notation Boole's Principles can be stated as:\footnote{
Boole's formulation of his Principles was incomplete for the application he had in mind, namely to his algebra of logic for classes. He did not discuss the possibility of restricting some of the laws so that they applied only to the class symbols, a restriction he quietly imposed on the idempotent law $x^2\equals x$. This restriction is handled by our $\delta(x)$.
} 
\begin{quote}
{\em
 If $\scrP \models \Sigma \cup \{(\forall x) \delta(x)\}$ and $\delta(x)$ is $\scrP$-total then
\begin{equation} \label{Boole's Principles}
(\forall \sigma \in \Horn) \big(\Sigma \vdash \sigma|_\delta\ \Rightarrow\ \scrP \models \sigma ).
\end{equation}
}
\end{quote}
Example \ref{CX} shows that this is false in general. The correct version is given in the next theorem, and $\S$\ref{applic} shows how it applies to Boole's algebra of logic for classes.

\begin{theorem} \label{principles}
Suppose  $\Sigma \cup \{(\forall x) \delta(x)\}$ is a set of laws for $\scrP$ with $\delta(x)$ a $\scrP$-total formula.   Then
\begin{equation} \label{goal}
(\forall \sigma \in \Horn) \big(\Sigma \vdash \sigma|_\delta\ \Rightarrow\ \scrP \models \sigma )
\end{equation}
iff
\begin{equation} \label{cond}
\scrP \,{\hookrightarrow}\, \Mod(\Sigma).
\end{equation}
\end{theorem}

\begin{proof}
($\Leftarrow$) Suppose that \eqref{cond} holds and $\Sigma \vdash \sigma|_\delta$. Given $\bP \in \scrP$ let $\bP {\hookrightarrow} \bQ \in \Mod( \Sigma)$. 
From $\bQ \in\Mod( \Sigma)$ it follows that $\bQ \models \sigma|_\delta$. Since  $\sigma|_\delta$ is a universal sentence, it follows that $\bP \models \sigma|_\delta$. Now $\delta$ is $\bP$-total and $\bP \models (\forall x) \delta(x)$, thus $\bP \models \sigma$.

($\Rightarrow$) For the converse suppose that \eqref{cond} fails. Then, for some $\bP\in \scrP$,
\begin{equation}
\Sigma \cup \Diag^+(\bP) \cup \Distinct(P) 
\end{equation}
is inconsistent, where $\Distinct(P)$ is the set of negated atomic formulas $p \napprox q$ for pairs $p,q$ of distinct elements of $P$.  By compactness there is a finite $P_0 := \{ p_1,\ldots,p_m\} \subseteq P$ and a finite subset $\scrS$ of $\Diag^+(\bP)$, with the only $p\in P$ mentioned in $\scrS$ being members of $P_0$, such that
\begin{equation}
\Sigma \cup \scrS \cup \Distinct(P_0) 
\end{equation}
is inconsistent.  
Since this is a set of Horn sentences, for some $i<j$, say $i=1$ and $j=2$,
we have
\begin{equation}
\Sigma \cup \scrS \cup \{ p_1 \napprox p_2 \} 
\end{equation}
is inconsistent. This is equivalent to
\begin{equation} \label{step0}
\Sigma \vdash \big(\bigwedge\scrS \Big)\  \rightarrow  \ p_1 \equals p_2.
\end{equation}
 We can assume, w.l.o.g, that $\delta(p) \in \scrS$ for $p \in P_0$. Thus we can write
 \eqref{step0} in the form
 \begin{equation} \label{step}
\Sigma \vdash \Big(\bigwedge\{ \delta(p) : p \in P_0\} \wedge \Omega(\bfp\,) \Big)\  \rightarrow \  p_1 \equals p_2,
\end{equation}
where $\Omega(\bfp\,)$ is a conjunction of atomic sentences from $\scrS$. 
This implies
 \begin{equation} \label{step2}
\Sigma \vdash (\forall \bfx\,) \Big[\Big(\bigwedge_{1 \le i \le m} \delta(x_i) \wedge \Omega(\bfx\,) \Big)\   \rightarrow\   x_1 \equals x_2 \Big].
\end{equation}
Letting
$$
\sigma := (\forall \bfx\,) \Big[ \Omega(\bfx\,)  \   \rightarrow \   x_1 \equals x_2 \Big]
$$
item \eqref{step2} becomes
 \begin{equation} \label{step3}
\Sigma \vdash \sigma|_\delta .
\end{equation}
But $\bP \not{\models}\, \sigma$ since $\bfp \in \Dom_\bP(\sigma)$, but it
fails to make the matrix of $\sigma$, namely
$$
 \Omega(\bfx\,)  \   \rightarrow \   x_1 \equals x_2,
$$
true. 
Thus the assumption that $\eqref{cond}$ fails leads to the fact that $\eqref{goal}$ fails.

\end{proof}

One can ask if there is a parallel result when working with identities.

\begin{question} \label{problem}
Suppose  $\scrP \models\Sigma$  where $\Sigma$ is a set of identities.   Does one have
\begin{equation} \label{goal2}
(\forall \sigma \in \Identities) \big(\Sigma \vdash \sigma\ \Rightarrow\ \scrP \models \sigma )
\end{equation}
iff
\begin{equation} \label{cond2}
\scrP \,{\hookrightarrow}\, \Mod(\Sigma)\ ?
\end{equation}
\end{question}


\section{Application to Boole's algebra of logic for classes} \label{applic}

Recall that Boole defined a partial algebra $\bP_U$ on the collection of subclasses of a universe $U\neq \O$ as follows:\footnote{
Note: To be in agreement with Boole's vocabulary  the word `class' is used here where the modern usage would usually be `set'.
 }
 
\begin{eqnarray*}
A\cdot B &:=& 
\begin{cases} 
A \cap B&\text{ for all  } A, B 
\end{cases}\\
A+B &:=& 
\begin{cases} 
A \cup B&\text{ if } A \cap B = \O\\ \text{undefined} & \text{otherwise}
\end{cases}\\
A-B &:=& 
\begin{cases} 
A \setminus B&\text{ if } B \subseteq A\\ \text{undefined} & \text{otherwise}
\end{cases}\\
1 &:=& U\\
0 &:=& \O
\end{eqnarray*}
Let $\scrP$ be the collection of $\bP_U$.

As mentioned earlier, Boole used, for the laws of $\scrP$, 
the usual equational  laws $CR_1$ for commutative rings with unity along with the law $0 \napprox 1$ and the quasi-identity laws that express `no additively nilpotent elements'. This set of universal Horn sentences will be denoted by $\Sigma$. Boole also used the idempotent law $x^2  \equals  x$, but in a limited sense---it only applied to class-symbols. Define $\delta(x)$ to be $x^2 \equals x$.

Note that $\scrP {\hookrightarrow} \Mod(\Sigma)$, namely one can easily embed $\bP_U$ into $\bbZ^U$, the ring of integers raised to the power $U$, by mapping $A\subseteq U$ to its characteristic function $\chi_A$.
In view of Theorem \ref{principles} it is not surprising that Hailperin used this fact to show that a large portion of Boole's algebra of logic for classes could be put on a firm basis---one has, from Theorem \ref{principles},
\begin{equation} \label{BP}
\Sigma \vdash \sigma|_\delta\quad \Rightarrow\quad \scrP \models \sigma.
\end{equation}

Boole was interested in applying usual equational reasoning, based on the laws $CR_1$ of numbers, to justify certain ground equational arguments:
\begin{equation}\label{arg}
\varepsilon_1(\tuple{A}\,), \ldots, \varepsilon_k(\tuple{A}\,)\quad \tfore\quad \varepsilon(\tuple{A}\,) .
\end{equation}
 He used the no-nilpotent-elements laws as inference rules (e.g., from $2t(\tuple{A}\,)\equals 0$ one has $t(\tuple{A}\,)\equals 0$).
The idempotent law only applied to the symbols $A_i$. 
Let
$$
\sigma := (\forall  \bfx)\Big[\Big(\bigwedge_{1\le i \le k} \varepsilon_i(\tuple{x} \,) \Big) \ \rightarrow\ \varepsilon(\tuple{x}\,)\Big].
$$
 To justify \eqref{arg} using Boole's ground equational reasoning is equivalent to showing $\Sigma \vdash \sigma|_\delta$; see Corollary 2.2.4 in Gorbunov \cite{Gorbunov-1998}.

So suppose \eqref{arg} follows from Boole's ground equational reasoning.
Then $\Sigma \vdash \sigma|_\delta$, so 
by $\eqref{BP}, \scrP \models \sigma$. This says \eqref{arg} is indeed valid in $\scrP$. 

\begin{example} \label{CX}
It is surprising that Boole formulated his Principles so generally since a counter-example was readily available from what he knew. With $\Sigma$ as above, let $\Sigma_1$ be $ \Sigma \cup (\forall x)(x^2\equals x)$. Then, with $\scrP$ being Boole's class of partial algebras, one has $\scrP \models \Sigma_1$.
However his Principles fail to hold since the following is false:
\begin{equation} \label{CXeq}
(\forall \sigma \in \Horn) \big(\Sigma_1 \vdash \sigma \ \Rightarrow\ \scrP \models \sigma \big).
\end{equation}
To see that \eqref {CXeq} is false, note that
$\Sigma_1\vdash (\forall x)\big( (2x)^2 \equals 2x\big)$, leading to  
$\Sigma_1 \vdash (\forall x) (2x \equals 0)$. Then $\Sigma_1 \vdash \sigma := (\forall x)( x \equals 0)$.  But clearly this $\sigma$ is not satisfied by $\scrP$.
\end{example}



\begin{thebibliography}{99}

\bibitem{Boole-1854}
George Boole,
{\em An Investigation of The Laws of Thought on Which are Founded the
Mathematical Theories of Logic and Probabilities}.
 Originally published by Macmillan,
London, 1854. Reprint by Dover, 1958.

\bibitem{Boole-1952}
George Boole, {\em
Collected Logical Works, Vol.~I, Studies in Logic and Probability}. Open Court, 1952.

 \bibitem{Burris-SEP}
 Stanley Burris,
 {\em George Boole}.
 The online Stanford Encyclopedia of Philosophy at 
 {\tt http://plato.stanford.edu/entries/boole/}.
 
  \bibitem{BuSa-1981}
S. Burris and H.P. Sankappanavar, {\em A Course in Universal Algebra}. Springer Verlag, 1981. The free, corrected version (2012) is available online as a PDF file at {\sf math.uwaterloo.ca/$\sim$snburris}.
 
  \bibitem{BuSa-2013}
\bysame,
  {\em The Horn theory of Boole's partial algebras.}  Bull. Symbolic Logic {\bf 19} (2013), no. 1, 97--105.
 
  \bibitem{BuSaII}
\bysame,
 {\em Boole's method I. A modern version.} ArXiv preprint.
 
 \bibitem{Gauss-1832}
 C.F. Gauss,
 {\em Theoria residuorum biquadraticorum}. Commentatio secunda., 
 Comm. Soc. Reg. Sci. G\"{o}ttingen {\bf 7} (1832), 1--34; 
 reprinted in {\em Werke}, Georg Olms Verlag, Hildesheim, 1973, pp. 93--148.
 
  \bibitem{Gorbunov-1998}
 Viktor A.~Gorbunov,  {\em Algebraic Theory of Quasivarieties}. Translated from the Russian. Siberian School of Algebra and Logic. Consultants Bureau, New York, 1998. xii+298 pp.
 
 \bibitem{G-G-1997}
 Grattan-Guinness and Bornet (eds.) {\em George Boole: Selected Manuscripts on Logic and Philosophy,}  Birkh\"{a}user, Basel-Boston-Berlin, 1997.
 
 \bibitem{Hailperin-1976}
 Theodore Hailperin,
{\em Boole's Logic and Probability, (Series: Studies in Logic and the Foundations of Mathematics, 85}, Amsterdam, New York, Oxford: Elsevier North-Holland. 1976. 2nd edition, Revised and enlarged, 1986.

\bibitem{Hamilton-1837}
William Rowan Hamilton, {\em Theory of Conjugate Functions, or
Algebraic Couples; with a Preliminary and Elementary Essay on Algebra as
the Science of Pure Time.} Transactions of the Royal Irish Academy {\bf 17} (1837),
293Ð422.

\bibitem{Peacock-1830}
 George Peacock, {\em Treatise on Algebra}. Cambridge, 1830. 2nd ed. 2 vols., Cambridge 1842/1845.
 
 \bibitem{Peacock-1833}
 \bysame, {\em Report on the Recent Progress and Present State of certain
Branches of Analysis}, presented at the 3rd meeting of the British
Association for the Advancement of Science, 1833.

\bibitem{Priest-1999} 
Graham Priest, Review of \cite{G-G-1997}, Studia Logica: An International Journal for Symbolic Logic, {\bf 63}, No. 1 (Jul., 1999), pp.~143-146.

\end{thebibliography}
\end{document}